\documentclass[10pt,reqno]{amsart}    

\usepackage{amssymb,amsmath,amscd}

\def\id{\mathop {\fam0 id}\nolimits}
\def\diag{\mathop {\fam0 diag}\nolimits}
\def\Cur{\mathop {\fam0 Cur}\nolimits}
\def\ConfAs{\mathop {\fam0 ConfAs}\nolimits}
\def\Cend{\mathop {\fam0 Cend}\nolimits}

\def\oo#1{\mathrel {{}_{(#1)}}}
\def\bo#1{\mathrel {{}_{[#1]}}}

\newtheorem{lemma}{Lemma}
\newtheorem{proposition}{Proposition}
\newtheorem{theorem}{Theorem}
\newtheorem{corollary}{Corollary}

\theoremstyle{definition}

\newtheorem{remark}{Remark}
\newtheorem{example}{Example}

\begin{document}

\title[On the Hochschild cohomologies of conformal algebras]{On the Hochschild cohomologies 
of associative conformal algebras with a finite faithful representation}

\author{P.S. Kolesnikov\and
       R.A. Kozlov}

\address{Sobolev Institute of Mathematics}
\email{pavelsk77@gmail.com}

\address{Novosibirsk State University}
\email{KozlovRA.NSU@yandex.ru}

\maketitle

\begin{abstract}
Associative conformal algebras of conformal endomorphisms are of essential importance
for the study of finite representations of conformal Lie algebras (Lie vertex algebras).
We describe all semisimple algebras of conformal endomorphisms
which have the trivial second Hochschild cohomology group 
with coefficients in every conformal bimodule. As a consequence, we state a complete
solution of the radical splitting problem in the class of associative conformal algebras
with a finite faithful representation.
\end{abstract}


 
\section{Introduction}

An algebraic formalization of the properties of 
the operator product expansion (OPE) in 
2-dimensional conformal field theory \cite{BPZ}
gave rise to a new class of algebraic systems,
vertex operator algebras \cite{Bor,FLM}.
The singular part of the OPE describes the commutator 
of two fields, and the corresponding
algebraic structures are called conformal (Lie) algebras \cite{Kac1996}
(or vertex Lie algebras \cite{FBZ2001}).

Namely, suppose $V$ is a vertex operator (super)algebra with a translation operator $\partial $
and a state-field correspondence $Y$.
Then, due to the locality axiom, the OPE of two fields
 $Y(a,z)$ and $Y(b,z)$, $a,b\in V$,
has a finite singular part:
$$
Y(a,w)Y(b,z) = \sum\limits_{n= 0}^{N(a,b)-1} Y(c_n, z) \frac{1}{(w-z)^{n+1}} 
+
(\mbox{regular part}).
$$
The coefficients of the singular part 
are completely determined by the (super)com\-mu\-tator 
of the fields:
$$
[Y(a,w),Y(b,z)] = \sum\limits_{n=0}^{N(a,b)-1}
Y(c_n,z)\frac{1}{n!} \frac{\partial^n\delta(w-z)}{\partial z^n},
$$
where 
$\delta(w-z) 
=\sum\limits_{s\in \mathbb Z} w^sz^{-s-1}$
is the formal delta-function.
The correspondence 
$$
(a,b)\mapsto c_n,\quad n\ge 0,
$$
defines an infinite series of bilinear operations ($n$-products) on $V$.
Together with the translation operator 
$\partial $, these operations 
turn $V$ into what is called a conformal Lie (super)algebra.

The most natural analogues of 
finite-dimensional algebras in the 
class of conformal algebras 
are finite ones, i.e., 
those finitely generated as modules over 
$H=\mathbb C [\partial ]$.

An algebraic study of this class of conformal algebras is an interesting mathematical problem
with numerous ties to other areas.
Structure theory of finite Lie conformal algebras was developed in \cite{DK1998}, 
simple and semisimple finite Lie conformal 
superalgebras were described in \cite{CK1997CK_6,FK2002,FK2004}.
Representations and cohomologies of conformal algebras were studied in \cite{BKV2001,BKL2011,BKL2013,CK1997,CKW1996,CKW2002,MZ2014}.

The study of universal structures for conformal algebras was initiated in
\cite{Roit1999}.  
The classical theory of finite-dimensional Lie algebras
often needs universal constructions like free algebras and universal enveloping algebras. 
This was a motivation for the development of combinatorial issues in the theory of
conformal algebras \cite{BFK2000,BFK2004}.  
 
One of the most intriguing questions in this field is related with the classical Ado Theorem. 
The latter states that every finite-dimensional Lie algebra has a faithful
finite-dimensional  representation. 
The Ado Theorem is a crucial point for understanding why every Lie algebra integrates 
globally into a Lie group. A formal approach to Lie theory (see \cite{SPM_Bull}) 
allows us to hope
that the ``fundamental triangle'' of Lie theory can be established for conformal algebras. 
To that end, an analogue of the Ado Theorem for conformal algebras will be required.
It was shown in \cite{Kol2011,Kol2016} that 
a finite (torsion-free) Lie conformal algebra
 has a faithful finite representation provided that its semisimple part splits as a 
 subalgebra, i.e., the analogue of the Levi Theorem holds. However, it is known that
the Levi Theorem does not hold for finite Lie conformal algebras in general 
(see, e.g., 
\cite{BKV2001,CKW2002}). 
In order to get further advance in the study of existence of faithful finite representations
we need to explore the structure of associative conformal envelopes of finite Lie algebras. 
These algebraic structures belong to the class of associative conformal algebras \cite{Kac1996} 
with a finite faithful representation (FFR, for short).

The structure theory of finite associative conformal algebras is very much similar
to ordinary associative algebras: simple objects 
are isomorphic to current conformal algebras $\Cur M_n(\mathbb C)$ over matrix algebras 
\cite{DK1998}, semisimple algebras are direct sums of simple ones, the maximal nilpotent
ideal (radical) always exists, and 
the semisimple part splits as a subalgebra \cite{Zelm2000}. 
Namely, 
for every finite associative conformal algebra $C$ there exists 
a maximal nilpotent ideal $R$ such that $E/R$ is isomorphic 
to the current conformal algebra $\Cur A$ 
over a semisimple finite-dimensional associative algebra~$A$. 
Moreover, the following analogue of the Wedderburn Principal Theorem holds: 
$C$ is isomorphic to the semi-direct product of $\Cur A$ and $R$, $C\simeq \Cur A\ltimes R$. 
The latter result also follows from the study of Hochschild cohomologies
for finite associative conformal algebras \cite{Dolg2007}, where it was shown that 
$H^2(\Cur A, M)=0$ for every semisimple algebra $A$, $\dim A<\infty$, and for every 
conformal bimodule $M$ over $\Cur A$.

Associative conformal algebras with a FFR form 
a more general class than (torsion-free) finite associative conformal algebras
\cite{Kol2011}. The conjecture on the structure of simple algebras in this class was 
posed in \cite{BKL2003} and proved in \cite{Kol2006_FFR}.
Semisimple associative conformal algebras with a FFR turn to be direct sums of simple ones, 
but was shown in \cite{Kol2008} that the analogue of the Wedderburn Principal Theorem 
does not hold in general. 

Since the splitting of a semisimple part plays crucial role in the study 
of Ado-type problems for finite Lie conformal algebras, it is reasonable to 
investigate the similar problem for associative conformal algebras with a FFR.
A natural tool for such investigation is the computation of Hochschild cohomologies
for conformal algebras. 
The latter were proposed in \cite{BKV2001}, but we prefer using the pseudo-tensor category approach of \cite{BDK2001}.
In this paper, we explicitly describe all those semisimple associative conformal algebras 
with a FFR that have trivial second Hochschild cohomology group relative to every 
conformal bimodule. 
The main technical statement is to show that conformal algebras of type  
$\Cend_{n,Q}$, where $Q=\diag\{1,\dots, 1,x\}$, always have trivial second 
cohomology group.
For $n=1$, it was done in \cite{Kozlov2017}. In this paper,
we use a different method of proof that does not work for $n=1$.  

Throughout the paper, 
$\Bbbk $ is an algebraically closed field of characteristic zero,
$\mathbb Z_+$ is the set of nonnegative integers.

\section{Preliminaries}

A conformal algebra \cite{Kac1996} is a linear space $C$ equipped with 
a linear map $\partial : C\to C$ and with a countable family 
of bilinear operations $(\cdot \oo{n} \cdot): C\otimes C\to C$, $n\in \mathbb Z_+$, 
satisfying the following axioms:
\begin{itemize}
 \item[(C1)] for every $a,b\in C$ there exists $N\in \mathbb Z_+$ such that $(a\oo{n}b)=0$ for all $n\ge N$;
 \item[(C2)] $(\partial a\oo{n}b)=-n(a\oo{n-1} b)$;
 \item[(C3)] $(a\oo{n} \partial b) = \partial (a\oo{n} b) + n(a\oo{n-1} b)$. 
\end{itemize}

Every conformal algebra $C$ is a left module over the polynomial algebra $H=\Bbbk[\partial ]$.
The structure of a conformal algebra on an $H$-module $C$ 
may be expressed by means of a single polynomial-valued map called $\lambda $-product:
\begin{equation}\label{eq:Lambda-prod}
 (\cdot \oo{\lambda }\cdot ): C\otimes C \to C[\lambda ], \quad 
 (a\oo{\lambda } b) = \sum\limits_{n\in \mathbb Z_+} \lambda^{(n)} (a\oo{n} b),
\end{equation}
where $\lambda $ is a formal variable, $\lambda^{(n)}=\lambda^n/n!$. The axioms (C2) and (C3) turn into 
the linearity property
\begin{equation}\label{eq:3/2-linearity}
 (\partial a\oo{\lambda } b) = -\lambda (a\oo{\lambda } b), \quad (a\oo{\lambda } \partial b) = (\partial +\lambda) (a\oo{\lambda } b).
\end{equation}

For every conformal algebra $C$ there exists a uniquely defined {\em coefficient algebra} $A(C)$ 
such that $C$ is isomorphic to a conformal algebra of formal distributions over $A(C)$ \cite{KacForDist1999}.
Conformal algebra $C$ is called associative (commutative, Lie, Jordan, etc.)
if so is $A(C)$ \cite{Roit1999}.
Every identity on $A(C)$ may be expressed as a family of identities on $C$. 
For example, $A(C)$ is associative 
if and only if 
\begin{equation}\label{eq:Conf_assoc(n)}
 (a\oo{n} (b\oo{m} c)) = \sum\limits_{s\ge 0} \binom{n}{s} ((a\oo{n-s} b)\oo{m+s} c)
\end{equation}
for all $a,b,c\in C$, $n,m\in \mathbb Z_+$. This family of identities may be expressed 
in terms of $\lambda $-product \eqref{eq:Lambda-prod} as
\begin{equation}\label{eq:Conf_assoc(lambda)}
 (a\oo{\lambda } (b\oo{\mu } c)) = ((a\oo{\lambda } b)\oo{\lambda +\mu } c), \quad a,b,c\in C,
\end{equation}
where $\lambda $ and $\mu $ are independent commuting variables \cite{KacForDist1999}.

A more conceptual approach to the theory of conformal algebras, their identities, representations, cohomologies, etc., 
is provided by the notion of a pseudo-algebra \cite{BDK2001}. 
Indeed, in ordinary algebra all basic definitions may be stated in terms of linear spaces, polylinear maps, and their compositions. 
For pseudo-algebras, the base field is replaced with a Hopf algebra $H$, 
the class of linear spaces is replaced with 
the class $\mathcal M(H)$ of left $H$-modules, and the role of $n$-linear maps is played by 
$H^{\otimes n}$-linear maps of the form
\[
\varphi: V_1\otimes \dots \otimes V_n \to H^{\otimes n}\otimes _H V, \quad V_i,V\in \mathcal M(H),
\]
where $H^{\otimes n}$ is considered as the outer product of regular right $H$-modules. 
Compositions of such maps are naturally defined by means of the expansion of $\varphi $ 
to an $H^{\otimes (m_1+\dots + m_n)}$-linear map
\begin{equation}\label{eq:pseudo-ext}
 (H^{\otimes m_1}\otimes_H V_1)\otimes \dots \otimes  (H^{\otimes m_n}\otimes_H V_n)\to H^{\otimes (m_1+\dots +m_n)}\otimes _H V, 
\end{equation}
$m_1,\dots , m_n\in \mathbb Z_+$, given by the following rule:
\[
 \varphi (1^{\otimes m_1}\otimes _H v_1,\dots, 1^{\otimes m_n}\otimes _H v_n) = 
 ((\Delta^{m_1}\otimes \dots \otimes \Delta^{m_n})\otimes _H \id_V)\varphi(v_1,\dots, v_n),
\]
where $\Delta^m : H\to H^{\otimes m}$ is the iterated coproduct on~$H$.

A pseudo-algebra is a left $H$-module $C$ equipped 
with a $H^{\otimes 2}$-linear map $*:C\otimes C\to H^{\otimes 2}\otimes_H C$ called pseudo-product, 
$*: a\otimes b \mapsto a*b$ (similar to the definition of an ordinary algebra as a linear space 
equipped with a bilinear product map).

Conformal algebras are exactly pseudo-algebras over the polynomial Hopf algebra $H=\Bbbk [\partial ]$ with coproduct 
$\Delta f(\partial) =f( \partial\otimes 1 + 1\otimes \partial )$, 
counit $\varepsilon(f(\partial))=f(0)$, and antipode $S(f(\partial) )=f(-\partial )$.
The relation between pseudo-product and conformal $\lambda $-product is given by 
\[
 a*b = (a\oo{\lambda }b)|_{\lambda =-\partial \otimes 1}.
\]
A conformal algebra $C$ satisfies \eqref{eq:Conf_assoc(n)} if and only if 
\[
 a*(b*c)=(a*b)*c, \quad a,b,c\in C,
\]
where $a*(b*c)\in H^{\otimes 3}\otimes _H C$ is the result of the composition $*(\id_C, *)$ on $a\otimes b\otimes c$
(the right-hand side is defined similarly).

\begin{remark}
The class $\mathcal M(H)$ together with $H$-polylinear maps and their compositions described above 
forms a non-symmetric pseudo-tensor category \cite{BD2004}. 
This language is enough to describe associative algebra features; to include 
commutativity into consideration (or any other relation involving permutations of variables)  
one needs a symmetric pseudo-tensor category structure on $\mathcal M(H)$ as described 
in \cite{BDK2001}. For example, the anti-commutativity identity in the language of pseudo-product 
is $a*b = -(\sigma_{12}\otimes_H \id_C)(b*a)$, where $\sigma_{12}$ acts on $H\otimes H$ as the permutation 
of tensor factors.
\end{remark}

Suppose $C$ is an associative conformal algebra considered as a pseudo-algebra over $H=\Bbbk[\partial ]$. 
A conformal bimodule over $C$ is a left $H$-module $M\in \mathcal M(H)$ equipped with $H^{\otimes 2}$-linear 
maps $l: C\otimes M \to H^{\otimes 2}\otimes _H M$ and $r: M\otimes C\to H^{\otimes 2}\otimes _H M$
satisfying three associativity identities:
\[
 l(*, \id_M)=l(\id_C, l),\quad l(\id_C,r)=r(l,\id_C) ,\quad  r(\id_M, *)=r(r,\id_C).
\]
In terms of ordinary operations, it means that we have two families of $M$-valued $n$-products 
defined on $C\otimes M$ and $M\otimes C$ satisfying the analogues 
of (C1)--(C3) and \eqref{eq:Conf_assoc(n)}.
We will also describe these infinite families by their generating functions denoted $(\cdot\oo{\lambda }\cdot)$.
Then
\[
 l(a,u) = (a\oo{\lambda} u)|_{\lambda = -\partial \otimes 1},\quad r(u,a) = (u\oo{\lambda} a)|_{\lambda = -\partial \otimes 1},
\]
and the analogues of 
\eqref{eq:3/2-linearity} and \eqref{eq:Conf_assoc(lambda)} hold.

Let us describe the Hochschild cohomology complex $\mathcal C^*(C,M)$ for an associative 
conformal algebra $C$ and 
a conformal bimodule $M$ over $C$ \cite{Dolg2007}. 
The space of $n$-cochains $\mathcal C^{n}(C,M)$ consists of all $H^{\otimes n}$-linear maps 
\[
 \varphi: C^{\otimes n}\to H^{\otimes n}\otimes _H M. 
\]
The conformal Hochschild differential 
$d_n: \mathcal C^{n}(C,M)\to \mathcal C^{n+1}(C,M)$ 
is defined similarly to the ordinary one, assuming the expansion \eqref{eq:pseudo-ext}:
\begin{multline}\label{eq:differential_pseudo}
 (d_n\varphi)(a_1,\dots, a_{n+1}) = a_1 * \varphi(a_2, \dots , a_{n+1}) \\
 + \sum\limits_{i = 1}^{n} (-1)^i\varphi(a_1,\dots , a_i * a_{i+1},\dots , a_{n+1}) 
  +  (-1)^{n+1}\varphi(a_1 ,\dots , a_n) * a_{n+1}.
\end{multline}
Denote by $\mathcal Z^n(C,M)$ and $\mathcal B^n(C,M)$ the subspaces of $n$-cocycles and $n$-co\-boun\-da\-ries, respectively. 
As for ordinary algebras, the quotient space $\mathcal H^n(C,M)=\mathcal Z^{n}(C,M)/\mathcal B^n(C,M)$ is called the $n$th Hochschild 
cohomology group of $C$ with coefficients in~$M$.

The complex $\mathcal C^*(C,M)$ may be described by means of $\lambda $-products.
For every $\varphi \in \mathcal C^n(C,M)$ and $a_1,\dots, a_n\in C$
we may write
\[
 \varphi(a_1,\dots , a_n)
  = \sum\limits_{s_1,\dots, s_{n-1}\in \mathbb Z_+} ( \partial^{(s_1)}\otimes \dots \otimes \partial^{(s_{n-1})}\otimes 1)
  \otimes _H u_{s_1,\dots , s_{n-1}},
\]
where $u_{s_1,\dots, s_{n-1}}\in M$ are uniquely defined.
Then one may consider a map 
\[
 \varphi_{\lambda_1,\dots, \lambda_{n-1}}: C^{\otimes n}\to M[\lambda_1,\dots, \lambda_{n-1}]
\]
defined by
\[
 \varphi_{\lambda_1,\dots, \lambda_{n-1}}(a_1,\dots, a_n) 
 = \sum\limits_{s_1,\dots, s_{n-1}\in \mathbb Z_+} (-1)^{s_1+\dots+s_{n-1}}  
 \lambda_1^{(s_1)} \dots \lambda_{n-1}^{(s_{n-1})} u_{s_1,\dots , s_{n-1}}.
\]
This map has the following sesquilinearity properties:
\[
\begin{gathered}
 \varphi_{\lambda_1,\dots, \lambda_{n-1}}(a_1,\dots,\partial a_i, \dots , a_n) = -\lambda_i \varphi_{\lambda_1,\dots, \lambda_{n-1}}(a_1,\dots, a_n),\quad i=1,\dots, n-1, \\
 \varphi_{\lambda_1,\dots, \lambda_{n-1}}(a_1,\dots, \partial a_n) = (\partial +\lambda_1 + \dots + \lambda_{n-1})\varphi_{\lambda_1,\dots, \lambda_{n-1}}(a_1,\dots, a_n).
\end{gathered}
\]
The differential \eqref{eq:differential_pseudo} turns into 
\begin{multline}\nonumber
(d_n\varphi)_{\lambda_1,\dots, \lambda_{n} }(a_1,\dots, a_{n+1}) = a_1 \oo{\lambda_1} 
\varphi_{\lambda_2,\dots, \lambda_n } (a_2, \dots , a_{n+1}) \\
 + \sum\limits_{i = 1}^{n} (-1)^i\varphi_{\lambda_1,\dots, \lambda_i+\lambda_{i+1},\dots , \lambda_n}
  (a_1,\dots , a_i \oo{\lambda_i} a_{i+1},\dots , a_{n+1}) \\
  +  (-1)^{n+1}\varphi_{\lambda_1,\dots, \lambda_{n-1}}(a_1 ,\dots , a_n) \oo{\lambda_1+\dots +\lambda_n} a_{n+1}. 
\end{multline}
For example, the space of 2-cocycles $\mathcal Z^2(C,M)=\mathrm{Ker}\,d_2\subset \mathcal C^2(C,M)$ consists of all sesquilinear maps 
$\varphi_\lambda : C\otimes C\to M[\lambda ]$ such that 
\begin{equation}\label{eq:Cocycle-lambda}
 a_1\oo{\lambda } \varphi_\mu (a_2,a_3) -\varphi_{\lambda +\mu}(a_1\oo\lambda a_2, a_3) 
 + \varphi_\lambda(a_1,a_2\oo\mu a_3) - \varphi_\lambda (a_1,a_2)\oo{\lambda+\mu} a_3 = 0.
\end{equation}

\begin{remark}
 It is easy to see that  $\mathcal C^*(C,M)$ coincides with the complex
 described in \cite{BKV2001}, where 
 $\mathcal C^n(C,M)$ consists of adjacent classes of sesquilinear maps 
 $\gamma_{\lambda_1,\dots, \lambda_n}: C^{\otimes n}\to M[\lambda_1,\dots, \lambda_n]$
 modulo the multiples of $(\partial+\lambda_1+\dots + \lambda_n)$.
 The correspondence is given by 
 \[
  \gamma_{\lambda_1,\dots, \lambda_n} \leftrightarrow \varphi_{\lambda_1,\dots, \lambda_{n-1}}
    =\gamma_{\lambda_1,\dots, \lambda_{n-1}, -\partial -\lambda_1-\dots - \lambda_{n-1}}.
 \]
\end{remark}

Recall that a null extension of an associative conformal algebra $C$ by means of a $C$-bimodule $M$ 
is an associative conformal algebra $E$
in a short exact sequence 
\[
 0\to M \to E \to C\to 0,
\]
where $(M\oo{\lambda } M)=0$ in $E$.
Two null extensions $E_1$ and $E_2$ are equivalent if there exists an isomorphism $E_1\to E_2$ such that the diagram 
\[
 \begin{CD}
0 @>>> M @>>> E_1 @>>> C @>>> 0\\
@. @V\id_M VV @VVV @VV\id_C V @. \\
0 @>>> M @>>> E_2 @>>> C @>>> 0
 \end{CD}
\]
is commutative.

\begin{theorem}[\cite{BKV2001,Dolg2007}]\label{thm:ExtCocycle}
Equivalence classes of null extensions of $C$ by means of $M$ are in one-to-one correspondence with $\mathcal H^2(C,M)$. 
\end{theorem}

A cocycle $\varphi \in \mathcal Z^2(C,M)$ corresponds to an extension $E=C\dot+ M$ with a new $\lambda$-product 
$(\cdot \oo{\hat\lambda }\cdot)$ given by
$ (u \oo{\hat\lambda } v)=0$ for $u,v\in M$,
$(a \oo{\hat\lambda }u) = (a \oo{\lambda }u)$,
$(u \oo{\hat\lambda }a) = (u \oo{\lambda }a)$
for $a\in C$, $u\in M$, and 
\[
 (a \oo{\hat\lambda }b) = (a \oo{\lambda }b) +\varphi_\lambda(a,b)
\]
for $a,b\in C$.

\begin{corollary}
Suppose $C$ is an associative conformal algebra such that 
$\mathcal H^2(C,M)=0$ for every $C$-bimodule $M$. 
Then $C$ splits in every extension with a nilpotent kernel. Namely, if $E$ 
is an associative conformal algebra with a nilpotent ideal $R$ such that 
$E/R\simeq C$ then $E\simeq C\ltimes R$.
\end{corollary}

\begin{example}
Let $A$ be an ordinary algebra, and let $H=\Bbbk[\partial ]$. 
Then the free $H$-module $H\otimes A$ equipped with a $\lambda $-product given by
 \[
  (h(\partial)\otimes a)\oo{\lambda }(g(\partial)\otimes b) = h(-\lambda )g(\partial +\lambda )\otimes ab,
  \quad h,g\in H,\ a,b\in A,
 \]
is called current conformal algebra $\Cur A$.
\end{example}

\begin{theorem}[\cite{Dolg2009}]\label{thm:H2Cur}
Let $A$ be a finite direct sum of matrix algebras over $\Bbbk $, $H=\Bbbk[\partial ]$.
Then $\mathcal H^2(\Cur A, M)=0$ for every conformal bimodule $M$ over $\Cur A$.
\end{theorem}

\begin{corollary}[\cite{Zelm2000}]
 Every finite associative conformal algebra 
 is a semi-direct sum of its nilpotent radical and semisimple part.
\end{corollary}

\begin{example}\label{exmp:Cend}
Consider the space $M_n(\Bbbk[\partial, x])$ equipped with the multiplication action of 
$\partial $ and with a $\lambda $-product 
\[
 A(\partial,x)\oo{\lambda } B(\partial , x) = A(-\lambda , x)B(\partial + \lambda, x+\lambda ).
\]
This is an associative conformal algebra denoted by $\Cend_n$.
\end{example}

The conformal algebra $\Cend_n$ has a finite faithful representation (FFR) on 
$M=\Bbbk[\partial ]\otimes \Bbbk^n\simeq \Bbbk^n[\partial ]$ given by 
\[
 A(\partial, x)\oo{\lambda } v(\partial ) = A(-\lambda , \partial )v(\partial+\lambda ).
\]
Every associative conformal algebra with a FFR is obviously isomorphic to a conformal 
subalgebra of $\Cend_n$ for an appropriate~$n$. 
Therefore, $\Cend_n$ plays the same role in the theory 
of conformal algebras 
as the matrix algebra $M_n(\Bbbk )$ does in the ordinary linear algebra. 
In \cite{Kol2011}, it was shown 
that every finite associative torsion-free conformal algebra has a FFR. 
However, $\Cend_n$ is infinite, so 
associative conformal algebras  with a FFR form a more general 
conformal analogue of the class of finite-dimensional algebras. 

\begin{theorem}[\cite{Dolg2009}]\label{thm:H2Cend}
One has $\mathcal H^2(\Cend_n, M)=0$
for every conformal bimodule $M$ over the conformal algebra $\Cend_n$. 
\end{theorem}

Conformal subalgebra $M_n(\Bbbk[\partial]) \subset \Cend_n$ is isomorphic to the current 
conformal algebra $\Cur M_n(\Bbbk )$ which is known to be simple \cite{DK1998}.
Given a matrix $Q=Q(x)\in M_n(\Bbbk[x])$, the set
$\Cend_{n,Q}=Q M_n(\Bbbk[\partial, x])$ is a conformal subalgebra (even a right ideal) 
of $\Cend_n$. If $\det Q\ne 0$ then $\Cend_{n,Q}$ is simple, and vice versa \cite{Kac1996}.

\begin{theorem}[\cite{Kol2006_FFR}]\label{thm:FFRsimple}
 Let $C$ be a simple associative conformal algebra with a FFR. 
 Then either $C\simeq \Cur M_n(\Bbbk )$ or 
 $C\simeq \Cend_{n,Q}$, $\det Q\ne 0$.
Semisimple associative conformal algebra with a FFR is a direct sum of simple ones.
\end{theorem}

It was shown in \cite{BKL2003} that, up to an isomorphism, one may assume $Q(x)$
is in the canonical diagonal form, i.e., 
\[
 Q(x) = \diag (f_1,\dots, f_n), \quad f_1\mid \dots \mid f_n.
\]
Moreover, if $\deg f_n>0$ then one may assume $f_n(0)=0$ since the shift map $x\mapsto x-\alpha$, $\alpha \in \Bbbk $, 
is an automorphism of $\Cend_n$.

\begin{remark}\label{rem:LeftRightCend}
 Note that $\Cend_{n,Q}$ is isomorphic as a conformal  algebra to 
 the left ideal $M_n(\Bbbk[\partial,x])Q(x-\partial)$ of $\Cend_n$ \cite{BKL2003}.
 It is easy to check by the definition of $\lambda $-product in Example~\ref{exmp:Cend}
 that the map $\theta : Q(x)A(\partial ,x)\mapsto A(\partial, x)Q(x-\partial )$, $A\in M_n(\Bbbk[\partial, x])$,
 is an isomorphism.
\end{remark}

Every associative conformal algebra $E$ with a FFR has a maximal nilpotent ideal 
$R$ such that $C=E/R$ also has a FFR. 
The conformal algebra $C$ obtained in this way is a direct sum of simple associative 
conformal algebras described by Theorem~\ref{thm:FFRsimple}. 
Theorems \ref{thm:H2Cur} and \ref{thm:H2Cend} imply $E = C\ltimes R$ if all 
summands in $C$ are of the form $\Cur_n$ or $\Cend_n$. 
In this paper, we consider all possible cases and explicitly determine those 
semisimple associative conformal algebras with a FFR that split 
in every extension with a nilpotent kernel.

\section{Extensions of $\Cend_{n,Q}$, $Q=\diag(1,\dots, 1, x)$}

Throughout this section, $C$ denotes the associative conformal 
algebra $\Cend_{n,Q}$ for $Q=\diag(1,\dots, 1, x)$. 
The main purpose of this section is to prove that for $n\ge 2$ 
we have $\mathcal H^2(C,M)=0$ for every conformal bimodule $M$ over~$C$.

\begin{proposition}\label{prop:DefRelCend_nQ}
The conformal algebra $C$
is generated by the set
\[
 X = \{e_{ij} \mid i=1,\dots, n-1, j=1,\dots, n\}\cup \{x_{ij} \mid i,j=1,\dots, n\}
\]
 relative to the following defining relations:
\begin{gather}
 e_{ij} \oo{\lambda}  e_{kl} = \delta_{jk} e_{il}, 
 \label{eq:DefRel1}\\
 x_{ij} \oo{\lambda} e_{kl} =  \delta_{jk} x_{il}, 
 \label{eq:DefRel2}\\
 e_{ij}{} \oo{\lambda} x_{kl} = \delta_{jk}(x_{il} +\lambda e_{il}), 
 \label{eq:DefRel3} \\
 x_{ij}\oo{1} x_{jl} = x_{il}, 
 \label{eq:DefRel4}\\
 x_{ij}\oo{\lambda } x_{kl} = 0, \quad j\ne k, 
 \label{eq:DefRel5}\\
 x_{ij}\oo{0} x_{jl} = x_{ik}\oo{0} x_{kl}. 
 \label{eq:DefRel6}
\end{gather}
\end{proposition}

\begin{proof}
Let $\ConfAs(X, 2)$ be the free associative conformal algebra
generated by $X$ with locality bound $N=2$ \cite{Roit1999}. 
Denote by $\ConfAs(X, 2\mid S)$ the quotient of $\ConfAs(X, 2)$ modulo 
the ideal generated by defining relations 
\eqref{eq:DefRel1}--\eqref{eq:DefRel6}.
Obviously, there is a homomorphism $\psi: \ConfAs (X,2\mid S)\to C$ sending $e_{ij}$ 
to the corresponding unit matrix 
 and $x_{ij}$ to $xe_{ij}$.  
To study $\ConfAs(X, 2\mid S)$, we apply the Gr\"obner---Shirshov bases 
technique for associative conformal algebras \cite{BFK2000,BFK2004}.

Assume the order on $\ConfAs(X,2)$ is induced as in \cite{BFK2000} by the following order on $X$:
$e_{ij}<x_{kl}$, 
$e_{ij}<e_{kl}$ or $x_{ij}<x_{kl}$ if and only if $(ij)<(kl)$ lexicographically.
The set of $S$-reduced words consists of 
\[
\begin{aligned}
& \partial ^s e_{ij},\quad i=1,\dots, n-1,\ j=1,\dots, n, \\
& \partial ^s x_{kl}, \quad k,l=1,\dots, n, \\
& \partial^s (x_{k1}\oo{0} \underset{t-1}{\underbrace{x_{11} \oo{0} \dots \oo{0} x_{11}}}\oo{0} x_{1l} ),
\quad k,l=1,\dots, n,
\end{aligned}
\]
$s\ge 0$, $t\ge 1$.

The images of these words under $\psi $ are linearly independent in $C$, so $S$ 
is a Gr\"obner---Shirshov basis of $C$ with respect to the generators~$X$.
\end{proof}

Let us now reduce the set of generators.

\begin{corollary}\label{cor:DefRelCend_nQ}
The conformal algebra $C$ is generated by 
$X' = \{e_{ij}, e_{1n}, x_{n1}\mid i,j=1,\dots,n-1 \}$ 
relative to the following 
defining relations:
\begin{gather}
 e_{ij}\oo{\lambda } e_{kl} = \delta_{jk} e_{il},  \quad i,j,k,l=1,\dots, n-1, \label{eq:DefRel2-1} \\
 e_{1n}\oo{\lambda } e_{ij} = 0, \quad i,j    =1,\dots, n-1,                   \label{eq:DefRel2-2} \\
 e_{11}\oo{\lambda } e_{1n} = e_{1n},                                          \label{eq:DefRel2-3} \\
 e_{ij}\oo{\lambda } x_{n1} = 0, \quad i,j=1,\dots, n-1,                       \label{eq:DefRel2-4} \\
 x_{n1}\oo{\lambda } e_{11} = x_{n1},                                          \label{eq:DefRel2-5} \\
 e_{1n}\oo{1} x_{n1} = e_{11},\quad e_{1n}\oo{m} x_{n1} = 0, \ m>1,            \label{eq:DefRel2-6} 
\end{gather}
\end{corollary}

\begin{proof}
Since $C$ is simple, it is enough to note that if \eqref{eq:DefRel2-1}--\eqref{eq:DefRel2-6} hold then the 
elements of $X'$ together with
\[
 \begin{gathered}
 e_{in} = e_{i1}\oo{0} e_{1n}, \quad i=2,\dots, n-1, \\
 x_{nj} = x_{n1}\oo{0} e_{1j}, \quad j=2,\dots, n, \\
 x_{ij} = e_{in}\oo{0}x_{nj}, \quad i=1,\dots, n-1,\ j=1,\dots, n
 \end{gathered}
\]
satisfy all relations \eqref{eq:DefRel1}--\eqref{eq:DefRel6}.
For example, let us check \eqref{eq:DefRel4}. First, 
\[
 x_{n1}\oo{m} e_{1l} = (x_{n1}\oo{0} e_{11})\oo{m} e_{1l} = 0
\]
for $m>0$ by \eqref{eq:DefRel2-5} and \eqref{eq:DefRel2-1}.
Next, it follows from conformal associativity that
\begin{multline}\nonumber
x_{ij}\oo{m} x_{jl} 
= (e_{i1}\oo{0}e_{1n}\oo{0} x_{n1}\oo{0} e_{1j})\oo{m} (e_{j1}\oo{0}e_{1n}\oo{0} x_{n1}\oo{0} e_{1l})\\
= (e_{i1}\oo{0}e_{1n}\oo{0} x_{n1}\oo{0} e_{1j}\oo{0} e_{j1})\oo{m}(e_{1n}\oo{0} x_{n1}\oo{0} e_{1l}) \\
= (e_{i1}\oo{0}e_{1n}\oo{0} x_{n1})\oo{m}(e_{1n}\oo{0} x_{n1}\oo{0} e_{1l}) \\
= (e_{i1}\oo{0}e_{1n}\oo{0} x_{n1}\oo{0}e_{1n})\oo{m} (x_{n1}\oo{0} e_{1l}) \\
= (e_{i1}\oo{0}e_{1n}\oo{0} x_{n1})\oo{0} \sum\limits_{s\in \mathbb Z_+}\binom{m}{s}((e_{1n}\oo{s} x_{n1})\oo{m-s} e_{1l}) \\
= (e_{i1}\oo{0}e_{1n}\oo{0} x_{n1})\oo{0} ((e_{1n}\oo{m} x_{n1})\oo{0} e_{1l})
= \begin{cases}
x_{il},    & m=1, \\
0,  & m>1.
  \end{cases}
\end{multline}
Other relations \eqref{eq:DefRel1}--\eqref{eq:DefRel6} can be checked in a similar way.
\end{proof}

Let $M$ be an arbitrary conformal bimodule over $C$.

\begin{lemma}\label{lem:L0}
 For every 2-cocycle $\varphi \in \mathcal Z^2(C,M)$ there exists 
 $\varphi'\in \mathcal Z^2(C,M)$ such that 
 $\varphi - \varphi' \in \mathcal B^2(C,M)$ 
 and
\begin{equation}\label{eq:CC-Cend0}
 \begin{gathered}
   \varphi'_\lambda (u_{ij},v_{kl})=0, \quad 1\le i,j,k,l\le n-1,\ u,v\in \{e,x\}.  
 \end{gathered}
\end{equation}
\end{lemma}

\begin{proof}
Note that the subalgebra $C_0\subset C$ generated by  
$u_{ij}$, $1\le i,j\le n-1$, $u\in \{e,x\}$,
is isomorphic to $\Cend_{n-1}$ (upper left block of size $n-1$ in $\Cend_{n,Q}$). 
The restriction of $\varphi $ on $C_0$ 
belongs to $\mathcal Z^{2}(C_0,M)$. By Theorem~\ref{thm:H2Cend}, $\mathcal H^2(C_0,M)=0$, 
so there exists 
$\tau \in \mathcal C^1(C_0,M)$ such that 
$(d_1\tau)_\lambda (u,v)=\varphi_\lambda (u,v)$ 
for all $u,v\in C_0$. Let us choose an arbitrary extension 
of $\tau $ to an $H$-linear map $C\to M$ and note 
that $\varphi' = \varphi -d_1\tau $ is the desired cocycle. 
\end{proof}

In the subsequent computations, we will use the following notation. 
For $a,b\in C$ and $\varphi\in \mathcal C^2(C,M)$, denote 
\[
 \{a\oo{\lambda } b\} = (a\oo{-\partial - \lambda } b)
 =\sum\limits_{n,s\in \mathbb Z_+} \frac{(-\lambda )^n}{n!} \frac{(-\partial)^s}{s!} (a\oo{n+s} b), 
\]
and, similarly,
\[
 \varphi_\lambda \{a,b\} = \varphi_{-\partial -\lambda }(a,b)= \sum\limits_{n,s\in \mathbb Z_+}
 \frac{(-\lambda )^n}{n!} \frac{(-\partial)^s}{s!} \varphi_{n+s}(a,b).
\]
It is well-known (see \cite{KacForDist1999}) that the following relation holds on an associative 
conformal algebra:
\begin{equation}\label{eq:RightMul_relation}
 a\oo{\lambda } \{b\oo{\mu } c \} = \{(a\oo{\lambda } b)\oo{\mu} c\}.
\end{equation}
Therefore, similar relations hold for conformal bimodule multiplications.

For a 2-cocycle $\varphi \in \mathcal Z^2(C,M)$, 
Theorem~\ref{thm:ExtCocycle}  and relation \eqref{eq:RightMul_relation} imply
\begin{equation}
 \varphi_\lambda(a, \{b\oo{\mu}c\}) + a\oo{\lambda}\varphi_{\mu}\{b,c\} = 
\varphi_{\mu}\{ (a_{\lambda}b), c\} + \{\varphi_\lambda(a, b)\oo{\mu}c\}.
\end{equation}
Since $\lambda $-product is sesquilinear, we also have 
\[
 (\varphi _\mu \{a,b\} \oo{\lambda } c ) = (\varphi_{\lambda-\mu}(a,b)\oo{\lambda }c)
\]
for all $a,b,c\in C$.

\begin{lemma}\label{lem:L1}
 For every 2-cocycle $\varphi \in \mathcal Z^2(C,M)$ there exists 
 $\varphi'\in \mathcal Z^2(C,M)$ such that 
 $\varphi - \varphi' \in \mathcal B^2(C,M)$, 
 \eqref{eq:CC-Cend0} holds,
 and
\begin{gather}
 \varphi'_\lambda (e_{1n}, e_{ij}) =0, \quad i,j=1,\dots, n-1,   \label{eq:EELambda-1}\\
 \varphi'_\lambda (e_{11}, e_{1n}) = 0.                          \label{eq:EELambda-2}
\end{gather}
\end{lemma}

\begin{proof}
Without loss of generality we may assume that $\varphi$ satisfies \eqref{eq:CC-Cend0}.
Denote $e=e_{11} + \dots + e_{n-1\, n-1} \in C_0\subset C$, 
where $C_0$ stands for the same subalgebra as in the proof of Lemma~\ref{lem:L0}.
Define $\tau \in \mathcal C^1(C,M)$
in such a way that $\tau(e_{1n}) = \varphi_0\{e_{11}, e_{1n}\} - \{e_{11}\oo{0} \varphi_0 \{e_{1n},e \}\}$ 
and $\tau(u) = 0$
for other generators $u$ of $C$ as of $H$-module.
Then, in particular, $\tau(C_0)=0$, so $(d_1\tau)_\lambda (C_0,C_0)=0$. 
Let us compute
\begin{multline}\nonumber
 (d_1\tau)_\lambda (e_{1n}, e_{ij}) 
= \varphi_\lambda (e_{11}, e_{1n})\oo\lambda e_{ij} - e_{11} \oo{\lambda } (\varphi_0(e_{1n}, e)\oo{0} e_{ij}) \\
= \varphi_\lambda(e_{11}, e_{1n}\oo{0} e_{ij}) + e_{11}\oo{\lambda } \varphi_0(e_{1n}, e_{ij}) - \varphi_\lambda (e_{11}\oo\lambda e_{1n}, e_{ij}) \\
- e_{11}\oo\lambda (\varphi_0(e_{1n}, e\oo{0} e_{ij}) + e_{1n}\oo{0} \varphi_0(e,e_{ij}) - \varphi_0(e_{1n}\oo{0} e, e_{ij})) \\
= -\varphi_\lambda (e_{1n}, e_{ij})
\end{multline}
for $i,j=1,\dots, n-1$, 
and
\begin{multline}\nonumber
(d_1\tau)_\lambda (e_{11}, e_{1n}) 
= -\tau (e_{1n}) + e_{11}\oo\lambda \tau(e_{1n})  \\
= -\varphi_0\{e_{11}, e_{1n}\} + \{e_{11}\oo{0} \varphi_0\{e_{1n}, e\}\} 
 + e_{11}\oo\lambda \varphi_0\{e_{11}, e_{1n}\} \\
- \{ (e_{11}\oo{\lambda } e_{11})\oo{0} \varphi_0\{e_{1n}, e\}\} 
= -\varphi_0\{e_{11}, e_{1n}\} + \varphi_0\{ e_{11}\oo\lambda e_{11}, e_{1n}\} \\
+ \{\varphi_\lambda (e_{11}, e_{11}) \oo{0} e_{1n}\}
- \varphi_\lambda (e_{11}, \{e_{11}\oo{0} e_{1n} \}) 
= -\varphi_\lambda (e_{11}, e_{1n}).
\end{multline}
Therefore, 
$\varphi' = \varphi + d_1\tau $ satisfies 
\eqref{eq:CC-Cend0}, \eqref{eq:EELambda-1}, and \eqref{eq:EELambda-2}.
\end{proof}

\begin{lemma}\label{lem:L2}
 For every 2-cocycle $\varphi \in \mathcal Z^2(C,M)$ there exists 
 $\varphi'\in \mathcal Z^2(C,M)$ such that 
 $\varphi - \varphi' \in \mathcal B^2(C,M)$, 
 \eqref{eq:CC-Cend0}, \eqref{eq:EELambda-1}, \eqref{eq:EELambda-2} hold,
 and
\begin{gather}
 \varphi'_\lambda (x_{n1}, e_{11}) = 0,  \label{eq:XELambda-1} \\
    \varphi'_\lambda( e_{ij}, x_{n1}) =0, \quad i,j=1,\dots, n-1.         \label{eq:XELambda-2}
\end{gather}
\end{lemma}

\begin{proof}
Define a 1-cochain $\tau \in \mathcal C^1(C,M)$ in such a way that 
\[
 \tau (x_{n1}) = \varphi_0(x_{n1}, e_{11}) -\varphi_0(e,x_{n1})\oo{0} e_{11}
\]
and $\tau(u)=0$ for other generators of $C$ as of $H$-module.
Then $(d_1\tau)_\lambda (C_0,C_0)= 0 $ and $(d_1\tau)_\lambda (e_{1n},C_0) = (d_1\tau)_\lambda (C_0,e_{1n}) = 0$.
Moreover, 
\begin{multline}\nonumber
 (d_1\tau)_\lambda (x_{n1}, e_{11})
 = \tau(x_{n1})\oo{\lambda} e_{11} - \tau (x_{n1})
 =\varphi_0(x_{n1}, e_{11})\oo{\lambda } e_{11} \\
 - \varphi_0(e,x_{n1})\oo{0}e_{11}\oo\lambda e_{11} 
 -\varphi_0(x_{n1}, e_{11}) + \varphi_0(e,x_{n1})\oo{0} e_{11} = \varphi_0(x_{n1}, e_{11}\oo{\lambda } e_{11}) \\
 + x_{n1}\oo{0}\varphi_\lambda (e_{11}, e_{11}) 
 -\varphi_\lambda (x_{n1}\oo{0} e_{11}, e_{11}) -\varphi_0(x_{n1}, e_{11})
 =-\varphi_\lambda (x_{n1}, e_{11})
\end{multline}
and 
\begin{multline}\nonumber
 (d_1\tau)_\lambda (e_{ij}, x_{n1}) = e_{ij}\oo{\lambda } \tau(x_{n1}) \\
 =e_{ij}\oo{\lambda } \varphi_0(x_{n1}, e_{11}) - (e_{ij}\oo{\lambda } \varphi_0(e,x_{n1})) \oo{\lambda } e_{11}
 =\varphi_\lambda (e_{ij}\oo\lambda x_{n1}, e_{11}) \\
 + \varphi_\lambda (e_{ij}, x_{n1})\oo\lambda e_{11}- \varphi_\lambda (e_{ij}, x_{n1}\oo{0} e_{11})
  - ( \varphi_\lambda (e_{ij}\oo\lambda e, x_{n1}) \\
  + \varphi_\lambda (e_{ij}, e)\oo\lambda x_{n1} 
    -\varphi_\lambda (e_{ij}, e\oo{0} x_{n1})  )\oo{\lambda } e_{11}
 = -\varphi_\lambda (e_{ij}, x_{n1}).
\end{multline}
Therefore, 
$\varphi' = \varphi + d_1\tau $ is the desired cocycle.
\end{proof}

\begin{lemma}\label{lem:L3}
 For every 2-cocycle $\varphi \in \mathcal Z^2(C,M)$ there exists 
 $\varphi'\in \mathcal Z^2(C,M)$ such that 
 $\varphi - \varphi' \in \mathcal B^2(C,M)$, 
 \eqref{eq:CC-Cend0}, \eqref{eq:EELambda-1}, \eqref{eq:EELambda-2}, 
 \eqref{eq:XELambda-1}, \eqref{eq:XELambda-2} hold,
 and
 \begin{equation}\label{eq:EXn1=0}
  \varphi'_\lambda (e_{1n}, x_{n1}) = \varphi'_0(e_{1n}, x_{n1}).
 \end{equation}
\end{lemma}

\begin{proof}
Given a 2-cocycle $\varphi \in \mathcal Z^2(C,M)$,
denote by $S$ the set of all $\varphi'\in \mathcal Z^2(C,M)$
such that $\varphi - \varphi' \in \mathcal B^2(C,M)$
and $\varphi'$ 
satisfies 
 \eqref{eq:CC-Cend0}, \eqref{eq:EELambda-1}, \eqref{eq:EELambda-2}, \eqref{eq:XELambda-1}, \eqref{eq:XELambda-2}.
Lemmas \ref{lem:L0}, \ref{lem:L1}, \ref{lem:L2} imply $S \ne \varnothing $.
Without loss of generality, we may 
assume $\varphi \in S$ and 
$m=\deg_\lambda \varphi_{\lambda}(e_{1n}, x_{n1})$ is minimal among all 
$\deg_\lambda \varphi'_{\lambda}(e_{1n}, x_{n1})$,
$\varphi'\in S$.
If $m=0$ then there is nothing to prove. 
If $m>0$ then define 
 $\tau(x_{n1}) = \frac{1}{m}x_{n1}\oo{0}\varphi_1(e_{1n}, x_{n1})$ and $\tau(u) = 0$
 for other $H$-linear generators $u$ of $C$.  
Let us compute  
\begin{multline}\nonumber
 \varphi_\mu(e_{1n}, x_{n1})\oo{\lambda } e_{11} 
  = \varphi_\mu (e_{1n}, x_{n1}\oo{\lambda -\mu } e_{11})   \\
   + e_{1n}\oo{\mu } \varphi_{\lambda -\mu}(x_{n1},e_{11})
   -\varphi_\lambda (e_{1n}\oo{\mu} x_{n1}, e_{11})
 =\varphi_{\mu}(e_{1n}, x_{n1})
\end{multline}
due to the choice of $S$.
Hence,
\begin{multline}\nonumber
 (d_1\tau )_\lambda (x_{n1}, e_{11}) 
  = \tau(x_{n1})\oo{\lambda } e_{11} - \tau(x_{n1}) \\
  =\frac{1}{m} x_{n1}\oo{0} ( \varphi_1(e_{1n},x_{n1})\oo\lambda e_{11} -  \varphi_1(e_{1n},x_{n1}))
  =0,
\end{multline}
\begin{multline}\nonumber
(d_1\tau )_\lambda (e_{ij}, x_{n1}) = e_{ij}\oo{\lambda } \tau(x_{n1})
= e_{ij}\oo{\lambda } \frac{1}{m} x_{n1}\oo{0} \varphi_1(e_{1n}, x_{n1}) \\
=\frac{1}{m} (e_{ij}\oo{\lambda } x_{n1})\oo{\lambda } \varphi_1(e_{1n}, x_{n1}) 
= 0.
\end{multline}
Therefore, the conditions of Lemmas \ref{lem:L1} and \ref{lem:L2} hold for $d_1\tau $,
so the cocycle $\varphi -d_1\tau $ belongs to the set $S$. 
Moreover, 
\begin{multline}\nonumber
(d_1\tau )_\lambda (e_{1n}, x_{n1})  = e_{1n}\oo{\lambda } \tau(x_{n1}) \\
 = \frac{1}{m} e_{1n}\oo{\lambda } (x_{n1}\oo{\lambda } \varphi_1(e_{1n}, x_{n1})) 
 = \frac{1}{m} (x_{11} +\lambda e_{11})\oo{\lambda } \varphi_1(e_{1n}, x_{n1}) \\ 
 = \frac{1}{m} (x\oo{0} e_{11}\oo{\lambda } \varphi_1(e_{1n}, x_{n1}) 
    + \lambda e_{11}\oo{\lambda } \varphi_1(e_{1n}, x_{n1}))
\end{multline}
has the same principal term (with respect to $\lambda $) as $\varphi_\lambda (e_{1n}, x_{n1})$.
Indeed, let us evaluate 
\begin{multline}\nonumber
 e_{11}\oo{\lambda } \varphi_\mu (e_{1n}, x_{n1}) 
  = \varphi_{\lambda+\mu}(e_{11}\oo\lambda e_{1n}, x_{n1})  \\
   + \varphi_\lambda (e_{11}, e_{1n})\oo{\lambda +\mu } x_{n1}
   - \varphi_\lambda (e_{11}, e_{1n}\oo\mu x_{n1}) 
   = \varphi_{\lambda +\mu}(e_{1n}, x_{n1})
   \in M[\lambda , \mu ].
\end{multline}
The coefficient at $\mu $ of the latter polynomial is equal 
$e_{11}\oo{\lambda } \varphi_1(e_{1n}, x_{n1})$.
If $\varphi_\lambda(e_{1n}, x_{n1}) = \lambda^{(m)}u_m + \lambda ^{(m-1)}u_{m-1} + \dots $
then 
\[
 e_{11}\oo{\lambda } \varphi_1(e_{1n}, x_{n1}) = \lambda ^{(m-1)}u_m +\lambda ^{(m-2)}u_{m-1} +\dots .
\]
Finally, 
\begin{multline}\nonumber
 (d_1\tau )_\lambda (e_{1n}, x_{n1}) 
 = \frac{1}{m}\lambda^{(m-1)} x\oo{0}u_m +\frac{1}{m}\lambda^{(m-2)} x\oo{0}u_{m-1}
 + \dots \\
 + \frac{1}{m}\lambda \lambda^{(m-1)} u_m + \frac{1}{m}\lambda \lambda^{(m-2)} u_{m-1} +\dots 
 = \lambda^{(m)}u_m + \dots ,
\end{multline}
and for $\varphi' = \varphi - d_1\tau \in S$ we have 
\[
\deg_\lambda \varphi'_\lambda (e_{1n}, x_{n1}) < m = \deg_\lambda \varphi_\lambda (e_{1n}, x_{n1})
\]
in contradiction to the choice of $\varphi $.
\end{proof}

\begin{theorem}\label{thm:H2Cend_nQ}
 Let $C=\Cend_{n,Q}$, $Q=\diag(1,\dots, 1, x)$, $n> 1$, 
 and let $M$ be an arbitrary conformal bimodule over $C$. 
 Then $\mathcal H^2(C,M)=0$.
\end{theorem}

\begin{proof}
Suppose $\varphi \in \mathcal Z^2(C,M)$. 
By Lemmas \ref{lem:L0}--\ref{lem:L3}, there exists $\varphi '\in \mathcal Z^2(C,M)$
such that $\varphi - \varphi'\in \mathcal B^2(C,M)$ and 
\eqref{eq:CC-Cend0}, \eqref{eq:EELambda-1}, \eqref{eq:EELambda-2}, 
 \eqref{eq:XELambda-1}, \eqref{eq:XELambda-2}, \eqref{eq:EXn1=0} hold.

Consider an extension $E$ defined by $\varphi' $ as in Theorem~\ref{thm:ExtCocycle}:
\[
 0\to M\to E\to C\to 0.
\]
Then $E=C \dot+ M$, and the pre-images of the elements of $X'$ from 
Corollary~\ref{cor:DefRelCend_nQ} satisfy defining relations 
\eqref{eq:DefRel2-1}--\eqref{eq:DefRel2-6}. Therefore, there exists 
a homomorphism $\rho : C\to E$ which maps $a\in X'$ to $a+0\in C\dot+ M=E$.
The map $\rho $ is injective since it has a left inverse $E\to C$ in the exact sequence above. 
Hence, the subalgebra of $E$
generated by pre-images of $X'$ is isomorphic to~$C$, the extension $E$ is split,
and $\varphi'\in \mathcal B^2(C,M)$. The latter implies $\varphi \in \mathcal B^2(C,M)$.  
\end{proof}

\begin{remark}
For $n=1$, the proof stated above does not work since there is no idempotent 
$e_{11}\in C$. 
However, it was proved in \cite{Kozlov2017} that 
$\mathcal H^2(\Cend_{1,x}, M)=0$ for every conformal bimodule $M$
over $\Cend_{1,x}$. As a corollary, Theorem~\ref{thm:H2Cend_nQ} holds also for $n=1$.
\end{remark}

\section{Non-split extensions with finite faithful representation}

In this section, we state a series of examples 
of non-split null extensions of semisimple associative conformal algebras with a FFR.

\begin{example}\label{exm:Cend+Cend}
Let $C=x\Bbbk[\partial , x]\oplus y\Bbbk[\partial ,y]\simeq \Cend_{1,x}\oplus \Cend_{1,x}$.
Consider $M=z^2\Bbbk [\partial , z]$ as a conformal bimodule over $C$ relative to 
\[
\begin{gathered}
xf(\partial , x)_{(\lambda )} z^2g(\partial , z) 
 = z^2(z+\lambda )f(-\lambda , z)g(\partial +\lambda , z+\lambda ), \\
z^2g(\partial , z)_{(\lambda )} y f(\partial , y) 
 = z^2(z+\lambda )g(-\lambda , z)f(\partial +\lambda , z+\lambda ), \\
yf(\partial , y)_{(\lambda )} z^2g(\partial , z) =0, \\
z^2g(\partial , z)_{(\lambda )} x f(\partial , x) = 0.
\end{gathered}
\]
\end{example}
Define a 2-cochain $\varphi \in C^2(C,M) $ in the following way:
\[
\begin{gathered}
\varphi_\lambda  (xf(\partial , x), y g(\partial , y) ) =  z^2f(-\lambda ,z)g(\partial +\lambda , z+\lambda ),\\
\varphi_\lambda  (yf(\partial , x), x g(\partial , x) ) = 0, \quad
\varphi_\lambda  (xf(\partial , x), x g(\partial , x) ) = 0, \\
\varphi_\lambda  (yf(\partial , y), y g(\partial , y) ) = 0.
\end{gathered}
\]

It is straightforward to check that $\varphi $ is a 2-cocycle. Indeed, one may either check 
the relation \eqref{eq:Cocycle-lambda}, or simply note that the set $E$ of all matrices of the form
\begin{equation}\label{eq:Cend+Cend(repr)}
\begin{pmatrix}
 xf(\partial , x) &\quad & \frac{1}{2}xf(\partial , x) + \frac{1}{2}(x-\partial ) g(\partial , x) + x(x-\partial )h(\partial , x) \\
 0 & \quad & (x-\partial ) g(\partial , x),
\end{pmatrix}
\end{equation}
$f,g,h \in \Bbbk [\partial , x]$,
is a conformal subalgebra of $\Cend_2$ isomorphic to the extension of $C$ by $M$ relative to $\varphi $, 
where the isomorphism is given by
\[
\begin{gathered}
xf(\partial , x) \mapsto  
 \begin{pmatrix} 
  xf(\partial , x) & & \frac{1}{2}xf(\partial , x) \\
  0 & & 0
 \end{pmatrix},
\quad 
y g(\partial , y) \mapsto 
\begin{pmatrix} 
  0 && \frac{1}{2}(x-\partial)g(\partial , x) \\
  0 && (x-\partial)g(\partial , x)
 \end{pmatrix}, \\
z^2h(\partial, z) \mapsto
\begin{pmatrix} 
  0 && x(x-\partial)h(\partial , x) \\
  0 && 0
 \end{pmatrix}.
\end{gathered}
\]

Let us show $\varphi \notin \mathcal B^2(C, M)$. Assume the converse, i.e., there exists 
$\psi \in \mathcal C^1(C,M)$
such that $d_1 \psi = \varphi $. Suppose 
\[
 \psi(x) = z^2f(\partial , z), \quad \psi(y) = z^2 g(\partial , z).
\]
Then 
\[
 z^2= \varphi_\lambda (x,y) = z^2(z+\lambda ) f(-\lambda , z) + z^2(z+\lambda) g(\partial +\lambda , z+\lambda), 
\]
i.e., $z+\lambda $ divides $1$ in $\Bbbk [\partial, z, \lambda ]$. The contradiction just obtained 
proves 
\[
\mathcal H^2(\Cend_{1,x}\oplus \Cend_{1,x}, M)\ne 0. 
\]

The example above clarifies the main idea of the following statement.

\begin{proposition}\label{prop:CendX+CendX}
 Suppose $C\simeq \Cend_{Q,n}\oplus \Cend_{Q',m}$, where
 $Q=\diag(\underset{n}{\underbrace{1,\dots ,1, x}})$, 
 $Q'=\diag(\underset{m}{\underbrace{1,\dots ,1, x}})$.
 Then there exists a conformal $C$-bimodule $M$ such that 
 $\mathcal H^2(C,M)\ne 0$.
\end{proposition}

\begin{proof}
Assume $n\le m$. Consider 
the $H$-module $M=M_{n,m}(\Bbbk[\partial, x])$
equipped with the following $C$-bimodule structure:
\[
 (Q(x)A(\partial , x) + Q'(x)B(\partial , x))\oo{\lambda } X(\partial , x) = A(-\lambda,x)Q(x+\lambda)X(\partial +\lambda , x+\lambda ),
\]
\[
 X(\partial , x)\oo{\lambda }(Q(x)A(\partial , x) + Q'(x)B(\partial , x)) = X(-\lambda , x)Q'(x+\lambda ) B(\partial+\lambda, x+\lambda ),
\]
$X\in M$. Straightforward computation shows that this is indeed a conformal bimodule over~$C$.

Let us define linear maps 
\[
 \vdash : M_{n}(\Bbbk[\partial,x])\to M_{n,m}(\Bbbk [\partial,x]), \quad \perp: M_m(\Bbbk[\partial,x])\to M_{n,m}(\Bbbk [\partial,x])
\]
as 
\[
 A^\vdash = \begin{pmatrix}
             A & 0
            \end{pmatrix}, 
\quad A\in M_n(\Bbbk[\partial,x])
\]
(add $m-n$ zero columns) and
\[
B^\perp = \begin{pmatrix}
           b_{11} & \dots & b_{1m} \\
           \hdotsfor3 \\
           b_{n1} & \dots & b_{nm}
          \end{pmatrix}, 
\quad B\in M_m(\Bbbk[\partial,x])
\]
(remove $m-n$ rows in the bottom).
It is clear that 
\[
 A^\vdash B = AB^\perp ,\quad  
 A_1A_2^\vdash =(A_1A_2)^\vdash , \quad 
 B_1^\perp B_2 = (B_1B_2)^\perp .
\] 

Consider the following 2-cochain $\varphi \in \mathcal C^2(C,M)$:
\begin{equation}\label{eq:Q+Q_Cochain}
 \varphi_\lambda (QA_1+Q'B_1, QA_2+Q'B_2) = A_1(-\lambda , x) B_2^\perp (\partial+\lambda, x+\lambda ).
\end{equation}
To prove that \eqref{eq:Q+Q_Cochain} defines a 2-cocycle, one may check \eqref{eq:Cocycle-lambda} 
for $a_i=QA_i+Q'B_i$, $i=1,2,3$. Indeed,
\begin{multline}\nonumber
 a_1\oo{\lambda } \varphi_\mu (a_2,a_3) 
 = (QA_1+Q'B_1)\oo{\lambda } \varphi_\mu (QA_2+Q'B_2, QA_3+Q'B_3) \\
 = QA_1 \oo{\lambda } A_2(-\mu , x)B_3^\perp (\partial+\mu, x+\mu ) \\
 = A_1(-\lambda , x)Q(x+\lambda )A_2(-\mu , x+\lambda )B_3^\perp (\partial+\lambda+\mu, x+\lambda+\mu )
\end{multline}
\begin{multline}\nonumber
 \varphi_{\lambda +\mu}(a_1\oo\lambda a_2, a_3) 
 = \varphi_{\lambda+\mu} (QA_1(-\lambda ,x)Q(x+\lambda )A_2(\partial+\lambda , x+\lambda )  \\
 +  Q'B_1(-\lambda ,x)Q'(x+\lambda )B_2(\partial+\lambda , x+\lambda ),   QA_3+Q'B_3) \\
 = A_1(-\lambda ,x)Q(x+\lambda )A_2(-(\lambda +\mu)+\lambda , x+\lambda )B_3^\perp (\partial+\lambda+\mu, x+\lambda+\mu ),
\end{multline}
\begin{multline}\nonumber
 \varphi_\lambda(a_1,a_2\oo\mu a_3)
 =\varphi_\lambda (QA_1+Q'B_1,  QA_2(-\mu ,x)Q(x+\mu )A_3(\partial+\mu , x+\mu ) \\
 +  Q'B_2(-\mu ,x)Q'(x+\mu )B_3(\partial+\mu , x+\mu ) ) \\
 = A_1(-\lambda, x)B_2^\perp (-\mu ,x+\lambda )Q'(x+\lambda+\mu )B_3(\partial+\lambda+\mu , x+\lambda+ \mu ),
\end{multline}
\begin{multline}\nonumber
 \varphi_\lambda (a_1,a_2)\oo{\lambda+\mu} a_3
 = A_1(-\lambda ,x)B_2^\perp(\partial+\lambda, x+\lambda ) \oo{\lambda+\mu} QA_3+Q'B_3 \\
 = A_1(-\lambda ,x)B_2^\perp(-(\lambda+\mu)+\lambda, x+\lambda )Q'(x+\lambda+\mu )B_3 (\partial+\lambda+\mu , x+\lambda+ \mu),
\end{multline}
so 
\[
 a_1\oo{\lambda } \varphi_\mu (a_2,a_3) = \varphi_{\lambda +\mu}(a_1\oo\lambda a_2, a_3),
 \quad 
 \varphi_\lambda(a_1,a_2\oo\mu a_3) = \varphi_\lambda (a_1,a_2)\oo{\lambda+\mu} a_3.
\]

Assume $\varphi = d_1\psi \in \mathcal B^2(C,M)$ for some $\psi \in \mathcal C^1(C,M)$. 
Suppose 
\[
 \psi(Q) = X_1(\partial, x),\quad \psi(Q'e_{nm})= X_2(\partial, x). 
\]
Then 
\begin{multline}\nonumber
 e_{nm} = I_n e_{nm}^\perp 
 =\varphi_\lambda(Q,Q') 
 = X_1\oo{\lambda }Q' - \psi(Q\oo{\lambda }Q') + Q\oo{\lambda }X_2 \\
 = X_1(-\lambda , x)Q'(x+\lambda ) - 0 + Q(x+\lambda )X_2(\partial+\lambda, x+\lambda),
\end{multline}
but in the right-hand side of the last expression a multiple of $x+\lambda $ occurs
in $n$th row and $m$th column.
Therefore, 
$\mathcal H^2(C,M)\ne 0$.
\end{proof}

\begin{remark}
The non-split extension
\[
 0\to M\to E\to C\to 0
\]
constructed with $C$, $M$, and $\varphi $ from Proposition \ref{prop:CendX+CendX}
has a FFR. 
\end{remark}

Indeed, one  may present $E$ as a subalgebra of  
$\Cend_{n+m}$ that consists of all matrices of the form
\[
\begin{pmatrix}
  Q(x)A & \quad & \frac{1}{2}Q(x)A^\vdash + \frac{1}{2} B^\perp Q(x-\partial ) + Q(x)XQ(x-\partial) \\
  0 & & BQ(x-\partial ) 
\end{pmatrix},
\]
where 
$A\in M_n(\Bbbk[\partial, x])$,
$B\in M_m(\Bbbk[\partial, x])$,
$X\in M_{n,m}(\Bbbk[\partial, x])$.

\begin{proposition}\label{prop:Cend2Q}
 Let $C = \Cend_{n,Q}$, $n>1$, and $Q=\diag(f_1,f_2,\dots , f_n)$, where 
 $f_i\in \Bbbk [x]$ and $f_1\mid f_2 \mid \dots \mid f_n$, $\det Q\ne 0$. 
 If there exist $1\le i<j\le n$ such that $\deg f_j\ge \deg f_i>0$ then 
 $\mathcal H^2(C)=\mathcal H^2(C,C)\ne 0$.
\end{proposition}

\begin{proof}
Let us choose a pair $i<j$ such that $0<\deg f_i\le \deg f_j$ and consider the 2-cochain given by
\begin{equation}\label{eq:Cend_n,Q-cocycle}
 \varphi_\lambda (Q(x)A, Q(x)B ) = Q(x)A(-\lambda , x)e_{ij} B(\partial+\lambda, x+\lambda ),
\end{equation}
where $e_{ij}$ stands for the corresponding unit matrix.

One may easily check that \eqref{eq:Cocycle-lambda} holds, so $\varphi \in \mathcal Z^2(C)$. 
Let us show 
$\varphi \notin \mathcal B^2(C)$. Indeed, assume there exists $\psi \in \mathcal C^1(C)$ 
such that $\varphi = d_1 \psi$.
Suppose
\[
 \psi(f_ie_{ii}) = Q(x) A(\partial, x), \quad \psi(f_je_{jj}) = Q(x)B(\partial, x)
\]
for some 
$A,B\in M_n(\Bbbk[\partial, x])$.
Let us compute
\[
 \varphi_\lambda (f_ie_{ii}, f_{j}e_{jj}) = Qe_{ii}e_{ij}e_{jj} = f_i e_{ij}
\]
and compare the result with 
\[
 (d_1\psi)_\lambda (f_ie_{ii}, f_{j}e_{jj}) = 
  Q(x)A(\partial, x)_{(\lambda )} f_{j}e_{jj} + f_ie_{ii} {}_{(\lambda )}Q(x)B(\partial, x).
\]
The latter is equal to 
\[
 f_j(x+\lambda )Q(x)A(-\lambda , x) e_{jj} + e_{ii}f_i(x) Q(x+\lambda )B(\partial+\lambda , x+\lambda), 
\]
so in $i$th row and $j$th column we get an equation
\[
 f_i(x) = f_j(x+\lambda )f_i(x)a_{ij}(-\lambda , x) + f_i(x) f_i(x+\lambda )b_{ij}(\partial+\lambda , x+\lambda), 
\]
which implies a contradiction since $f_i\mid f_j$. 
\end{proof}


\begin{remark}
The non-split extension 
$0\to C\to E\to C\to 0$ 
constructed with the cocycle from the proof 
of Proposition \ref{prop:Cend2Q} is a conformal algebra with a FFR. 
\end{remark}

Indeed, it is easy to see that
\[
E  \simeq 
\left\{
 \begin{pmatrix}
  Q(x) A & \quad & e_{ij}A+Q(x)B \\
  0 & & Q(x) A 
 \end{pmatrix}
 \mid 
 A,B \in M_n(\Bbbk[\partial, x])
\right\}\subseteq\Cend_{2n}.
\]

\begin{proposition}\label{prop:Cend_1f}
Let $C=\Cend_{1,f}$, $f=f(x)\in \Bbbk[x]$, $\deg f>1$, $n\ge 1$. 
Then there exists conformal bimodule $M$ over $C$ 
such that $\mathcal H^2(C,M)\ne 0$.
\end{proposition}

\begin{proof}
Linear shift of variable allows us to assume $f(x)=x^N + \alpha x^{N-2}+\dots $, i.e., 
there is no $x^{N-1}$, $N=\deg f$.

Consider $M=\Cend_1$ as a conformal bimodule 
over $C$ 
relative to the  following operations $(\cdot\bo\lambda \cdot)$:
\begin{equation}\label{eq:Cend1f-Action}
 fa \bo\lambda h = a \oo{\lambda } fh, \quad h\bo{\lambda } fa = h\oo{\lambda } fa,
\end{equation}
$a,h\in \Cend_1$,
where $(\cdot\oo\lambda \cdot)$ is the standard operation on $\Cend_1$ (see Example~\ref{exmp:Cend}).
This is indeed a bimodule since the right action of $C$ on $M$ coincides with the regular module structure, 
the left one is twisted 
by the isomorphism $\theta $  from Remark~\ref{rem:LeftRightCend}:
$(fa\bo{\lambda }h)=(\theta(fa)\oo{\lambda }h)$.

Define a sesquilinear map $\varphi : C\otimes C\to M[\lambda ]$ given by 
\begin{equation}\label{eq:Cend1f_Cocycle}
 \varphi_\lambda (fa,fb) = a\oo{\lambda }b. 
\end{equation}
It is straightforward to check that $\varphi \in \mathcal Z^2(C,M)$. 
Alternatively, one may note that $M$ is isomorphic to the radical of a conformal algebra $E$ with a FFR,
\[
E = \left\{\begin{pmatrix}
 f(x)a &\quad & f(x)a+f(x) h f(x-\partial) \\
 \noalign{\smallskip }
 0 & & a f(x-\partial)
\end{pmatrix} \mid a,h\in \Bbbk [\partial, x] \right\}
\subseteq \Cend_{2} .
\]
Namely, $h\in M$ corresponds to the matrix $f(x)h f(x-\partial )e_{12}$.

Let us show that $\varphi \notin \mathcal B^2(C,M)$.
Assume the converse: $\varphi = d_1\psi$ for some 
$\psi\in \mathcal C^1(C, M)$. 
Define $\tilde \psi : \Cend_1\to \Cend_1$ by the rule 
$\tilde \psi(a) = \psi(fa)$, $a\in \Cend_1$.
Then 
\[
 \varphi_\lambda (fa,fb) = \tilde\psi(a)\bo{\lambda } fb - \tilde\psi(a\oo{\lambda}fb) + fa\bo{\lambda }\tilde \psi(b)
\]
and 
\eqref{eq:Cend1f-Action}, \eqref{eq:Cend1f_Cocycle} imply
\begin{equation}\label{eq:Cend1f_coboundary}
 \tilde\psi(a\oo{n} fb) = \tilde\psi(a)\oo{n} fb + a\oo{n}f\tilde\psi(b) - a\oo{n} b, \quad n\in \mathbb Z_+.
\end{equation}
In particular, for $a=b=1$ we have 
$ \tilde\psi(1)\oo{N} f + 1\oo{N} f\tilde \psi(1) = \tilde\psi(1\oo{N} f) = N!\tilde\psi(1)$. 
Let us write $\tilde\psi(1)\in \Cend_1$ as
\[
 \tilde\psi(1) = \sum\limits_{k\ge 0} a_k(x) (x-\partial)^k, \quad a_k(x)\in \Bbbk [x].
\]
Then
\begin{multline}\label{eq:A_rel}
 N!  \sum\limits_{k\ge 0} a_k(x) (x-\partial)^k 
  =\sum\limits_{k\ge 0} a_k(x) (x-\partial)^k\oo{N} f(x) + 1\oo{N}  f(x)\sum\limits_{k\ge 0} a_k(x) (x-\partial)^k \\
  =\sum\limits_{k\ge 0} a_k(x) \frac{d^N}{dx^N}(x^kf(x)) + \sum\limits_{k\ge 0} \frac{d^N}{dx^N}(f(x)a_k(x)) (x-\partial)^k.
\end{multline}
Compare coefficients at $(x-\partial)^k$, $k\ge 0$, in the left- and right-hand sides of \eqref{eq:A-rel} to get
\begin{gather}
 N! a_0(x) = \sum\limits_{k\ge 0} a_k(x) \frac{d^N}{dx^N}(x^kf(x)) + \frac{d^N}{dx^N}(a_0(x) f(x)),      \label{eq:A-rel} \\
 N! a_k(x) = \frac{d^N}{dx^N}(f(x)a_k(x)),\quad k\ge 1.                                        \label{eq:B-rel}
\end{gather}
Relation \eqref{eq:B-rel} implies $\deg a_k \le 0$ for $k\ge 1$ (it is enough to compare principal terms of 
the polynomials). Then it follows from \eqref{eq:A-rel} that
\[
 0 = \frac{d^N}{dx^N}\left (\sum\limits_{k\ge 1}  a_k x^kf(x) + a_0(x) f(x) \right),
\]
i.e., 
\[
 a_0(x) = - \sum\limits_{k\ge 1}  a_k x^k.
\]
Hence, 
\[
 \tilde\psi(1) = \sum\limits_{k\ge 1}  a_k((x-\partial)^k- x^k),\quad a_k\in \Bbbk .
\]

It remains to show 
\begin{equation}\label{eq:B1}
\tilde\psi(h)= h\tilde\psi(1) 
\end{equation}
for all  $h\in \Cend_1$. It is enough to consider $h=h(x)\in \Bbbk[x]$
and proceed by induction on $\deg h$. Indeed, if \eqref{eq:B1} holds for some $h\in \Bbbk[x]$
then, by \eqref{eq:Cend1f_coboundary} for $a=h$, $b=1$, $n=N-1$, we have 
\begin{multline}\label{eq:coboundary_comp}
 N!\tilde \psi(xh) \\
 = h(x)\sum\limits_{k\ge 1}  a_k((x-\partial)^k- x^k)\oo{N-1} f(x) 
 + h\oo{N-1} f(x)\sum\limits_{k\ge 1}  a_k((x-\partial)^k- x^k)\\
 = h(x) \sum\limits_{k\ge 1} a_k \bigg(  \frac{d^{N-1}}{dx^{N-1}}(x^kf)- N!x^{k+1} 
 + N! x (x-\partial)^k - \frac{d^{N-1}}{dx^{N-1}}(x^k f) \bigg ) \\
 = N! xh(x)\tilde\psi(1).
\end{multline}
Here we used the initial assumption on the polynomial $f$ to get $f^{(N-1)}=N!x$.

Relation \eqref{eq:B1} means $\psi \in \mathcal B^1(C,M)$, so $d_1\psi =0\ne \varphi$.
\end{proof}

\begin{corollary}\label{cor:Cend_nQ}
Let $C = \Cend_{n,Q}$, $n\ge 1$, $Q=\diag(f_1,f_2,\dots , f_n)$, 
$\det Q\ne 0$. 
 If $\deg f_n>1$ then there exists a bimodule $M$ over $C$ such that 
 $\mathcal H^2(C,M)\ne 0$.
\end{corollary}

\begin{proof}
A non-split extension 
$0\to M\to E\to C \to 0$ 
may be constructed as 
\[
 E = \left\{
\begin{pmatrix}
 Q(x)A & \quad & Q(x)A+Q(x) X Q(x-\partial) \\
 \noalign{\smallskip }
 0 & & A Q(x-\partial)
\end{pmatrix} \mid A,X\in \Cend_n
\right\}\subseteq \Cend_{2n}.
\]
Indeed, $C$ contains a subalgebra $C'=e_{nn}Ce_{nn}$ isomorphic to $\Cend_{1,f_n}$, 
$M$ contains a $C'$-submodule $M'=e_{nn}M e_{nn}$.
The latter is isomorphic to 
$\Cend_1$ considered as a $\Cend_{1,f_n}$-bimodule relative to the operations
\eqref{eq:Cend1f-Action} for $f=f_n$, and the induced cocycle coincides with \eqref{eq:Cend1f_Cocycle}. 
If $E$ was a split extension then so is $E'=(e_{nn}+e_{2n\,2n})E(e_{nn}+e_{2n\,2n})$,
but $E'$ does not split by Proposition~\ref{prop:Cend_1f}.
\end{proof}

\section{Conclusion}

Let us analyze the general case in order to choose 
which semisimple associative conformal algebras with a FFR 
have trivial second Hochschild cohomology group 
relative to every conformal bimodule.

Let $C$ be a conformal algebra. 
Recall that a conformal identity (or unit) is an element $e\in C$
such that $e\oo\lambda e = e$ and $e\oo{0} a = a$ 
for all $a\in C$ \cite{Retakh2001}. 
Unital associative conformal algebras  and pseudo-algebras
are well studied, they are closely related with 
differential associative algebras \cite{Retakh2004}. 
From the cohomological point of view they also have nice properties.

\begin{proposition}\label{prop:SumWithUnit}
Suppose 
\[
  0\to M\to E \to C \to 0
\]
is an exact sequence of associative conformal algebras, $M$ is nilpotent, and 
$C= C_0\oplus C_1$, 
where 
$\mathcal H^2(C_1,M)=0$, 
$\mathcal H^2(C_0,M)=0$,
and 
$C_1$ is a unital conformal algebra. 
Then $\mathcal H^2(C,M)=0$ and, therefore, 
$E\simeq C\ltimes M$.
\end{proposition}

\begin{proof}
If $\bar e$ is a conformal identity of $C_1$ then there exists its pre-image $e\in E$ 
which is a conformal idempotent: $e\oo{0} e =e$, $e\oo{n} e = 0$ for $n>0$ \cite{Zelm2003,Dolg2007}.
Hence, we may apply conformal Pierce decomposition procedure to~$C$. In particular, 
the subalgebra $E_{0} = \{a - e\oo{0} a -\{a\oo{0} e\} + e\oo{0}\{a\oo{0} e\} \mid a\in E\}$
contains a pre-image of $C_0\subset C$. Therefore, $E_{0}$ contains a subalgebra 
isomorphic to $C_0$. On the other hand, the subalgebra 
$E_{1} = \{e\oo{0} \{a\oo{0} e\} \mid a\in E \}$ contains a subalgebra isomorphic to $C_1$. 
Finally, $E_{0}\oo{\lambda } E_{1} = E_{1}\oo{\lambda } E_{0}=0$ since 
$e\oo{m} E_0 = E_0\oo{m} e = 0$ for all $m\ge \mathbb Z_+$, so $E$ contains a subalgebra 
isomorphic to~$C$. 
\end{proof}

\begin{corollary}\label{cor:SplitRadical_Unit}
Let $C$ be an associative conformal algebra with a FFR,
and let $R$ be the maximal nilpotent ideal $R$, $C/R = \bigoplus_{i} C_i$, 
$C_i$ are simple conformal algebras.  
Suppose all $C_i$ are isomorphic to 
$\Cur_n$, $\Cend_n$, and no more than one of them is of the form 
$\Cend_{Q,n}$, $Q=\diag(1,\dots, 1,x)$.
Then $C\simeq C/R\ltimes R$.
\end{corollary}

Let us summarize Theorem~\ref{thm:H2Cend_nQ},
Corollary~\ref{cor:SplitRadical_Unit},
Example~\ref{exm:Cend+Cend},
Propositions~\ref{prop:CendX+CendX}, \ref{prop:Cend2Q}, \ref{prop:Cend_1f}, \ref{prop:SumWithUnit}
Corollary~\ref{cor:Cend_nQ},
and the results of \cite{Dolg2009,Kozlov2017} to state 
the ultimate description of those semisimple associative conformal algebras  with a FFR
that split in every extension with a nilpotent kernel.

By \cite{Kol2006_FFR}, every semisimple conformal algebra $C$ with a FFR 
is a direct sum of simple ones,
$C = \bigoplus_{i} C_i$, 
where either $C_i\simeq \Cur_n$ or $C_i\simeq \Cend_{n,Q}$, 
$Q=\diag(f_1,\dots , f_n)$, $\det Q\ne 0$, $f_1\mid \dots \mid f_n$.

\begin{theorem}\label{thm:Wedderburn}
An associative conformal algebra $C$ with  a FFR
splits in every extension with a nilpotent kernel
if and only if $C$ is a direct sum of conformal algebras $C_i$ isomorphic to 
$\Cur_n$, $\Cend_n$, and no more than one $\Cend_{n,Q}$, $Q=\diag(1,\dots, 1,x)$.
\end{theorem}

It is worth mentioning that if $C$ does not satisfy the condition of Theorem~\ref{thm:Wedderburn}
then there exists a non-split extension $E$ in an exact sequence $0\to M\to E\to C\to 0$
which is an conformal algebra with a FFR. Therefore, Theorem~\ref{thm:Wedderburn} may be 
considered as an analogue of the Wedderburn Principal Theorem for the class 
of associative conformal algebras with a FFR.

The list of those simple associative conformal algebras with a FFR that split 
in every null extension seems to be related with irreducible representations 
of finite Lie conformal superalgebras. Particular observations on simple superalgebras
lead to the following 

\medskip
\noindent
{\bf Conjecture.}
Let $L$ be a finite Lie conformal superalgebra, and let $M$ be a finite irreducible 
conformal $L$-module. Then the associative conformal subalgebra of $\Cend M$
generated  by the image of $L$ is isomorphic to either $\Cur_n$, or $\Cend_n$, or 
$\Cend_{n,Q}$ for $Q=\diag(1,\dots, 1, x)$, where $n$ is the rank of 
$M$ over $H=\Bbbk[\partial ]$.

\section*{Acknowledgements}
The work was supported by the Program of fundamental scientific researches of the Siberian Branch of 
Russian Academy of Sciences, I.1.1, project 0314-2016-0001.
The authors are grateful to the referees for useful comments that helped to improve the exposition.

\end{document}